\documentclass[10pt]{article}
\usepackage{amscd,amssymb,amsmath,latexsym,enumerate,amsthm,cite}
\usepackage[all]{xy}
\usepackage[mathscr]{euscript}
\usepackage{epsfig}
\usepackage{epsfig,psfrag}
\usepackage{float}
\usepackage{graphicx}

\usepackage{a4wide}

\usepackage{color,epsfig}
\title{Embeddings of self-similar ultrametric Cantor sets}
\author{Antoine Julien, Jean Savinien \\
 {\small Institut Camille Jordan, Universit\'e Lyon I, France}
 }

 \date{ }

\theoremstyle{plain}
\newtheorem{theo}{Theorem}[section]

\newtheorem{lemma}[theo]{Lemma}
\newtheorem{coro}[theo]{Corollary}
\newtheorem{hypo}[theo]{Hypothesis}

\theoremstyle{definition}
\newtheorem{defini}[theo]{Definition}
\newtheorem{exam}[theo]{Example}

\theoremstyle{remark}
\newtheorem{rem}[theo]{Remark}

\newtheoremstyle{citing}% name
  {\topsep}%      Space above, empty = `usual value'
  {\topsep}%      Space below
  {\itshape}% Body font
  {}%         Indent amount (empty = no indent, \parindent = para indent)
  {\bfseries}% Thm head font
  {.}%        Punctuation after thm head
  {.5em}%     Space after thm head: " " = normal interword space;
  {\thmnote{#3}}% Thm head spec

\theoremstyle{citing}
\newcommand{\Aa}{{\mathcal A}}
\newcommand{\Bb}{{\mathcal B}}

\newcommand{\Dd}{{\mathcal D}}
\newcommand{\Ee}{{\mathcal E}}

\newcommand{\Vv}{{\mathcal V}}

\newcommand{\CM}{{\mathbb C}}

\newcommand{\NM}{{\mathbb N}}

\newcommand{\RM}{{\mathbb R}}
\newcommand{\RMbar}{\overline{\mathbb R}}

\newcommand{\ZM}{{\mathbb Z}}

\newcommand{\TR}{{\rm Tr\,}}                       %%Trace
\newcommand{\Cs}{$C^{\ast}$-algebra }              %%C*-algebra with end space
\newcommand{\CsS}{$C^{\ast}$-algebras}             %%C*-algebras
\newcommand{\Sp}{\mbox{\rm Sp}}                    %%spectrum
\newcommand{\diam}{\mbox{\rm diam}}                %% diameter
\newcommand{\pf}{\Lambda}         %% Perron-Frobenius eigenvalue
\newcommand{\ext}{\text{\rm ext}_1}    %% extension edges one generation down

\newcommand{\eps}{\varepsilon}

\newcommand{\dmi}{\delta_{\min}}
\newcommand{\dma}{\delta_{\max}}
\newcommand{\mdix}{\mu}

\renewcommand{\phi}{\varphi}

\begin{document}

\maketitle

\begin{abstract}
\noindent
%This is a working version, it is not intended for general circulation.
We study self-similar ultrametric Cantor sets arising from stationary Bratteli diagrams.
We prove that such a Cantor set $C$ is bi-Lipschitz embeddable in $\RM^{[\dim_H(C)]+1}$, where $[\dim_H(C)]$ denotes the integer part of its Hausdorff dimension.
We compute this Hausdorff dimension explicitly and show that it is the abscissa of convergence of a {\it zeta}-function associated with a natural nerve of coverings of $C$ (given by the Bratteli diagram).
As a corollary we prove that the transversal of a (primitive) substitution tiling of $\RM^d$ is bi-Lipschitz embeddable in $\RM^{d+1}$.

We also show that $C$ is bi-H\"older embeddable in the real line.
The image of $C$ in $\RM$ turns out to be the $\omega$-spectrum (the limit points of the set of eigenvalues) of a Laplacian on $C$ introduced by Pearson-Bellissard {\it via} noncommutative geometry.
%This embedding can be derived from the Laplacian $\Delta_s$ that Pearson and Bellissard derived from the noncommutative geometry of $C$: the image of $C$ in $\RM$ is the $\omega$-spectrum of  $\Delta_s$ (the limit points of its set of eigenvalues).
\end{abstract}

%%%%%%%%%%%%%%%%%%

\tableofcontents

%%%%%%%%%%%%%%%%%%

%%%%%%%%%%%%%%%%%%%%%%%%%%%%%%%%%%%%%%%%%%%%%%%%%%%%%%%%%%%%%%%%%%%%
\section{Introduction and summary of the results}
\label{cantoremb10.sect-intro}

In this article, we study self-similar ultrametric Cantor sets, and their embeddings in Euclidean spaces. 
A prototype of such a space, and our motivating example, is the canonical transversal of a substitution tiling space.
%The prototype of a self-similar ultrametric Cantor set is the canonical transversal to a substitution tiling space.
The main motivation for this work was to derive embedding results for tilings tranversals.
We refer the reader to \cite{AP98,Rob04,Frank08}, as well as to \cite{JS10} Section 2, for basic notions about substitution tilings and tiling spaces.
Bratteli diagrams are important tools for the study of substitution tilings, and we use some of their techniques here.

%We work with ultrametrics, and our techniques are not transferable {\it as is} to the study of (simply) metric Cantor sets.
%Although ultrametrics are very natural for tilings transversals and totally disconnected spaces in general, this is a very strong requirement (and restriction).

Ultrametrics are very natural for tilings transversals as well as totally disconnected spaces in general.
The condition for a metric to be a ultrametric, see equation \eqref{cantoremb10.eq-ultra}, is a very strong requirement: for instance if two ultrametric balls intersect then one contains the other.
Our techniques, and results, are not transferable {\it as is} to the study of general metric Cantor sets.
The proofs of our embedding theorems are fairly elementary in comparison to what is known for metric spaces.
% (think of Assouas Theorem for instance, \cite{BS07} section 8.1).
The reader can compare for example with Assouad theorem for metric spaces with the doubling property (see \cite{BS07} section 8.1).

%A  compact metric space \((X,\rho)\) satisfies the {\em doubling property} if there exists an integer \(n \in\NM\), such that for all \(x \in X\) and \(r>0\), the ball of radius $r$ about $x$ can be covered by $n$ balls of radius $r/2$: there exist \(x_1, x_2, \cdots x_n \in X\) such that
%%%%%%%%%%%%%%%%%%%%%%
%\[
%B(x,r) \subset \bigcup_{j=1}^n B(x_j, r/2)
%\]
%%%%%%%%%%%%%%%%%%%%%%
%
%We now recall Assouad's Theorem which gives a criterion for H\"older embeddability of a compact metric space in a finite dimensional Euclidean space.
%%%%%%%%%%%%%%%%%%%%%%
%\begin{theo}
%\label{cantoremb10.thm-Assouad}
%Let \((X,\rho)\) be a compact metric space satisfying the doubling property.
%Given \(s\in (0,1)\), there exists an integer \(n_s \in \NM\) such that \((X,\rho^s)\) is bi-Lipschitz embeddable in \(\RM^{n_s}\).
%\end{theo}
%%%%%%%%%%%%%%%%%%%%%%

%%%%%%%%%%%%%%%%%%%%%%%%%%%%%%%%%%%%%%%%%%%%%%%%%%%%%%%%%%%%%%%%%%%%
\paragraph{Ultrametric Cantor sets and Bratteli diagrams}
A {\em ultrametric Cantor set} $(C,\rho)$ is a compact, Hausdorff, perfect (no isolated points), and totally disconnected space $C$, equipped with a metric $\rho$ which satisfies a stronger form of triangle inequality, namely
%%%%%%%%%%%%%%%%%%%%%
\begin{equation}
\label{cantoremb10.eq-ultra}
\forall z \in C\,, \quad \rho (x,y) \leq \max \{ \rho(x,z), \rho(z,y) \}
%\rho(x,y) \le \sup_{z\in C} \bigl( \rho(x,z), \rho(z,y) \bigr)\,.
\end{equation}
%%%%%%%%%%%%%%%%%%%%%
By Michon's theorem \cite{Mich85} there is an isometric equivalence between ultrametric Cantor sets and weighted Cantorian trees.
Given such a Cantor set, the ultrametric allows to define a nerve of partitions by clopen sets (closed and open sets) which gives rise to a Cantorian tree: its vertices represent the clopens and the weight encode their diameters.
And conversely, any such weighted Cantorian tree defines a ultrametric Cantor set.

The authors showed in \cite{JS10} that one can equivalently use {\em Bratteli diagrams} (Definition~\ref{cantoremb10.def-bratteli}) instead of Cantorian trees: paths in the diagram (like vertices of the graph) encode the clopen sets.
The formalism with Bratteli diagrams turns out to be very handy for {\em self-similar} Cantors which correspond to {\em stationary} Bratteli diagrams.
A self-similar Cantor is then viewed as an iterated function system (IFS) directed by the graph given by the Bratteli diagram (see Remark~\ref{cantoremb10.rem-reg}).

%The prototype of a self-similar ultrametric Cantor set is the canonical transversal to a substitution tiling space.
%The main motivation for this work was to derive embedding results for such spaces.
%We refer the reader to \cite{Frank08, Rob04}, as well as to \cite{JS10}, section 2, for basic notions about substitution tilings and tiling spaces.

%%%%%%%%%%%%%%%%%%%%%%%%%%%%%%%%%%%%%%%%%%%%%%%%%%%%%%%%%%%%%%%%%%%%
\paragraph{The {\it zeta}-function}
Consider a self-similar ultrametric Cantor set $(C,\rho)$.
Its Bratteli diagram encodes a nerve of partitions by clopens $(\Pi_n)_{n\in \NM}$.
For each $n$, $\Pi_n$ is viewed as a set of paths $\gamma$ of length $n$ in the diagram, whose corresponding clopen sets $[\gamma]$ partition $C$.
There is a natural {\it zeta}-function associated with $(C,\rho)$ (Definition~\ref{cantoremb10.def-zeta}), namely the complex power series:
%%%%%%%%%%%%%%%%%%%%%
\[
\zeta(s) = \sum_{n=0}^\infty \sum_{\gamma \in \Pi_n} \diam_\rho ([\gamma])^{s} \,.
\]
%%%%%%%%%%%%%%%%%%%%%
When it exists, its {\em abscissa of convergence} is the real $s_0>0$ such that $\zeta(s)$ is holomorphic in the half-plane \(\Re(s)>s_0\), and singular at  $s_0$.
It turns out that $\zeta$ is exactly the {\it zeta}-function of the spectral triple that Pearson and Bellissard built for ultrametric Cantor set \cite{PB09,JS10}.
In noncommutative geometry (NCG) \cite{Co94}, this abscissa of convergence is interpreted as the dimension of the noncommutative space (here the \Cs of continuous functions on $C$ with the sup norm).
In Lemma~\ref{cantoremb10.lem-s0}, we show that $s_0<+\infty$, and identify $s_0$ explicitly in terms of the contractions of the IFS.
We further prove the following.

\vspace{.1cm}
\noindent {\bf Theorem~\ref{cantoremb10.thm-hausdorff}}
{\em 
The Hausdorff dimension of a self-similar ultrametric Cantor set is equal to the abscissa of convergence of its {\it zeta}-function.
}
\vspace{.1cm}

As a corollary of this and previous results in \cite{JS10} we compute the Hausdorff dimension of the transversal of a (primitive) substitution tiling of $\RM^d$.

\vspace{.1cm}
\noindent {\bf Theorem~\ref{cantoremb10.thm-zeta}}
{\em 
Consider a primitive substitution tiling of $\RM^d$, with canonical transversal $\Xi$.
The abscissa of convergence $s_0$ of the {\it zeta}-function of $\Xi$ is equal to its Hausdorff dimension \(\dim_H(\Xi)\), and moreover one has
%%%%%%%%%%%%%%%%%%%%%
\[
s_0 =\dim_H(\Xi) = d\,.
\]
%%%%%%%%%%%%%%%%%%%%%
}
%\vspace{.1cm}

%%%%%%%%%%%%%%%%%%%%%%%%%%%%%%%%%%%%%%%%%%%%%%%%%%%%%%%%%%%%%%%%%%%%
\paragraph{Embeddings and the Pearson-Bellissard's Laplacians}

We prove several embedding results for self-similar ultrametric Cantor sets: Lipschitz embedding in $\RM^n$, and H\"older embedding in $\RM$.
Our main result is the following.

\vspace{.1cm}
\noindent {\bf Theorem~\ref{cantoremb10.thm-hoelder2}}
{\em
Let \((C,\rho)\) be a self-similar ultrametric Cantor set.
There exists a bi-Lipschitz embedding
%%%%%%%%%%%%%%%%%%%%%
\[
C \hookrightarrow \RM^{[\dim_H(C)]+1} \,,
\]
%%%%%%%%%%%%%%%%%%%%%
where $\dim_H(C)$ is the Hausdorff dimension of $C$, and $[\dim_H(C)]$ denotes its integer part.
}
\vspace{.1cm}

And as a corollary of this and Theorem~\ref{cantoremb10.thm-zeta}, we prove an embedding of tilings transversals.

\vspace{.1cm}
\noindent {\bf Theorem~\ref{cantoremb10.thm-bilipXi}}
{\em
The transversal of a substitution tiling of $\RM^d$ is bi-Lipschitz embeddable in $\RM^{d+1}$.
}
\vspace{.1cm}

We also prove a H\"older embedding on the real line.

\vspace{.1cm}
\noindent {\bf Theorem~\ref{cantoremb10.thm-Hoelder}}
{\em
A self-similar ultrametric Cantor set is bi-H\"older embeddable in the real line.
}
\vspace{.1cm}

This is interesting in the case of the transversal $\Xi$ of a substitution tiling space of $\RM^d$.
The Pearson-Bellissard spectral triple allows to define a one-parameter family $(\Delta_s)_{s\in\RM}$ of Laplace-Beltrami-like operators on $\Xi$ (as ``squares'' of the Dirac operator).
Those operators were introduced in \cite{PB09} and studied in details in \cite{JS10}.
Under some techniqueal but fairly general assumptions (Remark~\ref{cantoremb10.rem-tech}) we prove that their spectra ``contains'' the transversal as follows.

\vspace{.1cm}
\noindent {\bf Corollary~\ref{cantoremb10.cor-holderdim}}
{\em
For all $s$ greater than \(2(d + 1)\), the $\omega$-spectrum of $\Delta_s$ (the limit points of its pure point spectrum) is bi-H\"older homeomorphic to $\Xi$.
}
%\vspace{.1cm}

%%%%%%%%%%%%%%%%%%%%%%%%%%%%%%%%%%%%%%%%%%%%%%%%%%%%%%%%%%%%%%%%%%%%
\section{Self-similar Cantor sets and stationary Bratteli diagrams}
\label{cantoremb10.sect-bratteli}

Bratteli diagrams were first introduced in \cite{Bra72} to classify $AF$ \CsS.
They provide also a handy formalism for substitutions \cite{For97,DHS99,JS10}, and tilings in general \cite{BJS10}.

%%%%%%%%%%%%%%%%%%%%%
\begin{defini}
\label{cantoremb10.def-bratteli}
A {\em Bratteli diagram} is an infinite oriented graph \(\Bb = (E,V)\) where the sets of edges $E$ and vertices $V$ are given by the disjoint unions
%%%%%%%%%%%%%%%%%%%%%
\[
E= \coprod_{n\in \NM} E_n \,, \qquad V=\coprod_{n\in\NM} V_n\,,
\]
%%%%%%%%%%%%%%%%%%%%%
where $E_n$ and $V_n$ are finite sets, \(V_0 = \{ \circ\}\) is called the root of the diagram, and $E_0$ is isomorphic to  $V_1$.
The range and source maps are defined as \(r: E_n \rightarrow V_{n}\), \(s: E_n \rightarrow V_{n-1}\) .
The integer $n$ is called the {\em depth} in the diagram.

A {\em path} is a sequence \(\gamma=(e_0, e_1, e_2, \ldots) \) of composable edges: \(e_i \in E_i\), \(r(e_i)=s(e_{i+1})\).
One lets $\Pi_n$ denote the sets of paths down to depth $n$, and for $\gamma \in \Pi_n$ one call $|\gamma|=n$  its {\em length}.
Also, one sets $\Pi=\cup_n \Pi_n$ for the set of finite paths, and $\Pi_\infty$ for the set of infinite paths. 
For $\gamma \in \Pi_m$, $m\in \NM\cup\{ +\infty\}$, and $1\le n < m$, one sets $\gamma_{[0,n]}$ for the restriction of $\gamma$ to $\Pi_n$.
One extends the range map to $\Pi$ as follows: if \(\gamma=(e_0, e_1, \ldots e_n)\in \Pi_n\) then \(r(\gamma) = r(e_n)\).

The diagram is {\em stationary} if the sets $E_n$ and $V_n$ are respectively pairwise isomorphic for all $n\ge 1$.
One then writes \(E_n \cong \Ee\), \(V_n \cong \Vv\), $n\ge 1$.
In this case, the diagram is entirely determine by its {\em adjacency matrix}
%%%%%%%%%%%%%%%%%%%%%
\begin{equation}
\label{cantoremb10.eq-adjmatrix}
A_{vv'} = \# \{ e \in \Ee\, : \, s(e) = v\,, r(e) = v' \}\,, \ v,v' \in \Vv\,.
\end{equation}
%%%%%%%%%%%%%%%%%%%%%
\end{defini}
%%%%%%%%%%%%%%%%%%%%%

With the product topology of the $E_n$, $n\in \NM$, the set $\Pi_\infty$ is compact and totally disconnected.
A base of clopen sets (closed and open sets) for this topology is given by the {\em cylinders}
%%%%%%%%%%%%%%%%%%%%%
\begin{equation}
\label{cantoremb10.eq-cylinder}
[\gamma] = \{ x\in \Pi_\infty \; : \; x_{[0,n]} = \gamma\}\,, \quad \gamma \in \Pi_n, \, n\in \NM\,.
\end{equation}
%%%%%%%%%%%%%%%%%%%%%
If $\Pi_\infty$ has no isolated point, then it is a Cantor set.
This is the case whenever the following condition is satisfied.
%%%%%%%%%%%%%%%%%%%%%
\begin{hypo}
\label{cantoremb10.hyp-cantor}
For any $n\in\NM$ and any $v\in V_n$, there exists $m> n$ such that for any $v'\in V_m$ there is a path from $v$ to $v'$.
\end{hypo}
%%%%%%%%%%%%%%%%%%%%%

%%%%%%%%%%%%%%%%%%%%%
\begin{defini}
\label{cantoremb10.def-weight}
A {\em weight function} on a Bratteli diagram \(\Bb=(E,V)\), is a function  \(w : \Pi \rightarrow (0, +\infty)\) which satisfies
%%%%%%%%%%%%%%%%%%%%%
\begin{enumerate}[(i)]

\item \(w(\gamma) \le w(\gamma_{[0,m]})\) for any $\gamma \in \Pi_n$ and  any $m< n$,

\item \(\sup \{ w(\gamma) \; : \; \gamma \in \Pi_n\} \rightarrow 0\) as $n\rightarrow +\infty$.
\end{enumerate}
%%%%%%%%%%%%%%%%%%%%%
A {\em weighted Bratteli diagram}, is a Bratteli diagram equipped with a  weight function.
\end{defini}
%%%%%%%%%%%%%%%%%%%%%

%A {\em ultrametric} $\rho$ is a metric which satisfies the following strong form of triangle inequality:
%%%%%%%%%%%%%%%%%%%
%\[
%\rho(x,y) \le \sup_z \bigl( \rho(x,z), \rho(z,y) \bigr)\,.
%\]
%%%%%%%%%%%%%%%%%%%
%As a consequence, if two ultrametric balls intersect, then one contains the other.

A weight function $w$ on a Bratteli diagram \(\Bb=(E,V)\) defines a ultrametric $\rho_w$ on $\Pi_\infty$ as follows
%%%%%%%%%%%%%%%%%%%%%
\begin{equation}
\label{cantoremb10.eq-diamcyl}
%\diam_{w}([\gamma]) = w(\gamma) \,, \qquad \rho_w(x,y) = w(x\wedge y)\,, \; x \neq y \in \Pi_\infty\,,
\diam_{w}([\gamma]) = w(\gamma) \,, \qquad 
\rho_w(x,y) = \left\{
\begin{array}{ccc}
w(x\wedge y) & \textrm{if} & x \neq y \\
0 & \textrm{if} & x=y
\end{array} \right.
%x,y \in \Pi_\infty\,,
\end{equation}
%%%%%%%%%%%%%%%%%%%%% 
where $x\wedge y$ is the longest common prefix of $x$ and $y$.

In this manner any weighted Bratteli diagram satisfying Hypothesis~\ref{cantoremb10.hyp-cantor}. defines a ultrametric Cantor set.
The converse is a result of Michon \cite{Mich85} (rephrased here in terms of Bratteli diagrams).
%%%%%%%%%%%%%%%%%%%%%
\begin{theo}
\label{cantoremb10.thm-michon}
For any ultrametric Cantor set $C$, there exists a weighted Bratteli diagram satisfying Hypothesis~\ref{cantoremb10.hyp-cantor}, whose set of infinite path is isometric to $C$.
%There is an isometric equivalence between ultrametric Cantor sets and weighted Bratteli diagrams satisfying Hypothesis~\ref{cantoremb10.hyp-cantor}.
\end{theo}
%%%%%%%%%%%%%%%%%%%%%

One has the obvious following Lemma.
%%%%%%%%%%%%%%%%%%%%%
\begin{lemma}
\label{cantoremb10.lem-bilipsscantor}
Let $w,w'$ be two weights on a Bratteli diagram, and assume that there exists constants \(0<c_1<c_2\) such that 
\(c_1 w'(\gamma) \le w(\gamma) \le c_2 w'(\gamma)\),
for all \(\gamma \in \Pi\).
Then the ultrametrics associated with $w$ and $w'$ are bi-Lipschitz equivalent, namely one has: \(c_1 \rho_{w'} \le \rho_w \le c_2 \rho_{w'}\).
\end{lemma}
%%%%%%%%%%%%%%%%%%%%% 

Michon's theorem and Lemma~\ref{cantoremb10.lem-bilipsscantor} motivate the following definition.
%%%%%%%%%%%%%%%%%%%%%
\begin{defini}
\label{cantoremb10.def-sscantor}
A self-similar ultrametric Cantor set is the set \(\Pi_\infty\) of infinite paths in a stationary
Bratteli diagram with ultrametric $\rho$ given, for some \(\alpha \in (0,1)\) and 
\(a_\gamma \in (0,+\infty), \, \gamma\in \Pi\), by: 
%%%%%%%%%%%%%%%%%%%%%
\[
\rho(x,y) = a_\gamma \alpha^{|x\wedge y|}\,, \qquad \text{and the condition} \qquad
0 < \inf_{\gamma \in \Pi} a_\gamma \le \sup_{\gamma \in \Pi} a_\gamma < + \infty \,.
\]
%%%%%%%%%%%%%%%%%%%%%
A self-similar ultrametric Cantor set is called {\em regular} if \(a_\gamma = 1\)
for all \(\gamma \in \Pi\).
\end{defini}
%%%%%%%%%%%%%%%%%%%%%

%%%%%%%%%%%%%%%%%%%%%
\begin{rem}
\label{cantoremb10.rem-reg}
%%%%%%%%%%%%%%%%%%%%%
\begin{enumerate}[(i)]

\item In view of Lemma \ref{cantoremb10.lem-bilipsscantor}, any self-similar ultrametric Cantor set is
bi-Lipschitz equivalent to a regular one.
If one writes \(a= \inf_\gamma a_\gamma\) and \(a'=\sup_\gamma a_\gamma\), then the two constants 
\(c_1, c_2\), in Lemma \ref{cantoremb10.lem-bilipsscantor} can be taken to be $a/a'$ and $a'/a$ respectively.

\item A regular self-similar Cantor set as defined above can be seen as the invariant set of a graph directed IFS.
The maps are in one-to-one correspondence with the edges in $\Ee$, and have all contraction factor $\alpha$.
The directing graph is the graph of adjacencies of $\Ee$, and its paths correspond to $\Pi$.
\end{enumerate}
%%%%%%%%%%%%%%%%%%%%%

\end{rem}
%%%%%%%%%%%%%%%%%%%%%

%%%%%%%%%%%%%%%%%%%%%
\begin{exam}
\label{cantoremb10.ex-substitution}
Given a substitution on a finite alphabet $\Aa$ (that is a map which to each letter of $\Aa$ associates a finite word over $\Aa$), define the \emph{Abelianization matrix}
%%%%%%%%%%%%%%%%%%%%%
\begin{equation}
\label{cantoremb10.eq-abelmatrix}
A_{ij} =  \textrm{ number of occurences of $j$ is the substitution of $i$.}
\end{equation}
%%%%%%%%%%%%%%%%%%%%%
Then one can associate stationary Bratteli diagram with $\Vv \cong \Aa$ and adjacency matrix equal to $A$.
Under some conditions like primitivity (see below) and border forcing, see~\cite{Kel95}, it turns out that there is a canonical map between $\Pi_\infty$ and the subshift of $\Aa^\ZM$ associated with the substitution.
%And one identifies \(\Pi_\infty\) with the symbolic transversal $\Xi$ of the substitution.

A square matrix with non-negative entries is {\em primitive} if some power $A^n$ has positive entries.
When $A$ is primitive, then \(\Pi_\infty\) satisfies Hypothesis~\ref{cantoremb10.hyp-cantor} and is thus a Cantor set, and the the subshift of $\Aa^\ZM$ associated with the substitution is minimal
Also, $A$ has a Perron-Frobenius eigenvalue $\pf$: it is real, simple, and of modulus greater than those of the other eigenvalues.
Here $\pf$ is also greater than one, and is the dilation factor of the substitution.
If one normalizes the Perron-Frobenius eigenvector, such that its coordinates add up to 1, the latter are the frequencies of the letters: \(\nu_i, i\in \Vv\).
A natural choice of ultrametric is then given by the following weights:
%%%%%%%%%%%%%%%%%%%%%
\begin{equation}
\label{cantoremb10.eq-metricsimbsub}
w(\gamma) = \nu_{r(\gamma)} \Lambda^{-|\gamma|}\,, \quad \quad  \text{\it i.e.\ } \quad \quad
a_\gamma = \nu_{r(\gamma)} \,, \quad \alpha = \pf^{-1}\,.
\end{equation}
%%%%%%%%%%%%%%%%%%%%% 
In the case of substitution tilings of $\RM^d$, in order to recover the metric generally used on the canonical transversal $\Xi$ of the tiling space, one can take the following
%%%%%%%%%%%%%%%%%%%%%
\begin{equation}
\label{cantoremb10.eq-metricsubRd}
a_\gamma = (\nu_{r(\gamma)})^{1/d} \,, \quad \alpha = \pf^{-1/d}\,.
\end{equation}
%%%%%%%%%%%%%%%%%%%%% 
With such ultrametric, \(\Pi_\infty\) is bi-Lipschitz homeomorphic to $\Xi$ (see \cite{JS10} section 2).
\end{exam}

%%%%%%%%%%%%%%%%%%%%%

Given a Bratteli diagram $\Bb=(E,V)$, and an integer $k\ge 1$, one defines its $k$-th {\em telescoping} $\Bb^{(k)}$ as follows: 
%%%%%%%%%%%%%%%%%%%%%
\[
V^{(k)}_n \cong V_{kn} \qquad 
%E^{(k)}_n \cong \bigl\{ \gamma \in \Pi \, : \, s(\gamma) \in V_{(n-1)k-1}, \, r(\gamma) \in V_{nk}\bigr\}\,,
E^{(k)}_n \cong \; \text{\rm set of paths from } \; V_{(n-1)k-1} \ \text{\rm down to } \;V_{nk} \,.
\]
%%%%%%%%%%%%%%%%%%%%%
If the ultrametric of $\rho$ on \(\Pi_\infty\) has parameter $\alpha$, then one sees that the inherited ultrametric $\rho^{(k)}$ on \(\Pi_\infty^{(k)}\) has parameter $\alpha^k$.
Also if $A$ is the adjacency matrix of $\Bb$, then $A^k$ is that of $\Bb^k$.

%%%%%%%%%%%%%%%%%%%%%
\begin{lemma}
\label{cantoremb10.lem-telescoping}
Let $\Bb=(E,V)$ be a stationary Bratteli diagram, and \((\Pi_\infty,\rho)\) its associated self-similar Cantor set, with regular ultrametric $\rho$ of parameter $\alpha$.
Then \((\Pi_\infty,\rho)\) and \((\Pi_\infty^{(k)},\rho^{(k)})\) are bi-Lipschitz equivalent for all $k\ge 1$, namely one has:
%%%%%%%%%%%%%%%%%%%%%
\[
 \alpha^k \rho^{(k)} < \rho \le \rho^{(k)} \,.
\]
%%%%%%%%%%%%%%%%%%%%%
\end{lemma}
%%%%%%%%%%%%%%%%%%%%% 
\begin{proof}
Let \(x,y \in \Pi_\infty\) with \(|x\wedge y| = n\) so that \(\rho(x,y) = \alpha^n\).
Write \(n = q k + r\) with \(0\le r < k\).
We have \(\rho^{(k)}(x,y) = (\alpha^k)^q\).
Hence we get  \( \rho(x,y) = \alpha^{qk +r} \le \alpha^{qk}=\rho^{(k)}(x,y)\) and \( \alpha^k \rho^{(k)}(x,y) = \alpha^{k + (n-r)} < \alpha^n = \rho(x,y)\).
This completes the proof.
\end{proof}

%%%%%%%%%%%%%%%%%%%%%
\begin{defini}
\label{cantoremb10.def-zeta}
The {\em zeta}-function of a self-similar utrametric Cantor set  \((\Pi_\infty,\rho)\) is the (formal) complex series
%%%%%%%%%%%%%%%%%%%%%
\[
\zeta(s) = \sum_{n=0}^\infty \sum_{\gamma \in \Pi_n} \diam_\rho ([\gamma])^{s} \,.
\]
%%%%%%%%%%%%%%%%%%%%%
When it exists, its {\em abscissa of convergence} is the real $s_0>0$ such that $\zeta(s)$  is holomorphic in the half-plane \(\Re(s)>s_0\), and 
singular at  $s_0$.
\end{defini}
%%%%%%%%%%%%%%%%%%%%%

%%%%%%%%%%%%%%%%%%%%%
\begin{lemma}
\label{cantoremb10.lem-s0}
Let \((\Pi_\infty,\rho)\) be a self-similar ultrametric Cantor set.
Let $\alpha$ be the parameter of $\rho$, and $\pf$ the Perron-Frobenius eigenvalue of the adjacency matrix as in equation~\eqref{cantoremb10.eq-adjmatrix}.
The abscissa of convergence of the {\it zeta}-function exists and is given by
%%%%%%%%%%%%%%%%%%%%%
\begin{equation}
\label{cantoremb10.eq-s0}
s_0 = \frac{\log(\pf)}{-\log(\alpha)} \,.
\end{equation}
%%%%%%%%%%%%%%%%%%%%%
\end{lemma}
%%%%%%%%%%%%%%%%%%%%%
\begin{proof}
The {\it zeta}-function of \((\Pi_\infty,\rho)\) reads
%%%%%%%%%%%%%%%%%%%%%
\[
\zeta(s) = \sum_{n=0}^\infty \sum_{\gamma\in\Pi_n} a_\gamma^s \;  \alpha^{ns} \,,
\]
%%%%%%%%%%%%%%%%%%%%%
where the $a_\gamma$ are the parameters of $\rho$ as in Definition~\ref{cantoremb10.def-sscantor}.
Set $c_1 = \inf_\gamma a_\gamma$ and $c_2= \sup_\gamma a_\gamma$.
For all $s>0$ we have the inequalities
%%%%%%%%%%%%%%%%%%%%%
\[
c_1^s \sum_{n=0}^\infty \#(\Pi_n) \;  \alpha^{ns} 
\; \le \; \zeta(s)
\; \le \;
c_2^s \sum_{n=0}^\infty \#(\Pi_n) \;  \alpha^{ns} \,.
%c_1^s \tilde{\zeta}(s) \; \le \;  
%\zeta(s)  \; \le \; c_2^s  \tilde{\zeta}(s) \,, 
%\quad \text{\rm where} \quad
%\tilde{\zeta}(s) = \sum_{n=0}^\infty \#(\Pi_n) \;  \alpha^{ns} \,.
\]
%%%%%%%%%%%%%%%%%%%%%
%Therefore the abscissa of convergence $s_0,\tilde{s}_0$, of the two {\it zeta}-functions $\zeta, \tilde{\zeta}$, must agree: $s_0=\tilde{s}_0$.
%In other words we could have assumed that $\rho$ was regular.
Now \(\#(\Pi_n)\) grows like $\pf^n$ with $n$, that is, there are (uniform) constants \(0<c'_1<c'_2\) such that  
%%%%%%%%%%%%%%%%%%%%%
\[
c'_1 \pf^{n} \le \#(\Pi_n) \le c'_2 \pf^{n}\,. 
\]
%%%%%%%%%%%%%%%%%%%%%
Therefore one gets the inequalities
%%%%%%%%%%%%%%%%%%%%%
\[
c_1^s c'_1 \sum_{n=0}^\infty (\pf \alpha^s)^n \le \zeta(s) \le c_2^s c'_2 \sum_{n=0}^\infty (\pf \alpha^s)^n\,,
\]
%%%%%%%%%%%%%%%%%%%%%
and one sees that the power series is divergent for \(\Re(s)= - \log(\pf)/ \log(\alpha)\), and absolutely convergent for \(\Re(s)> - \log(\pf)/ \log(\alpha)\).
%Hence \(s_0 = \tilde{s}_0 = - \log(\pf)/ \log(\alpha)\).
\end{proof}

%%%%%%%%%%%%%%%%%%%%%
\begin{theo}
\label{cantoremb10.thm-hausdorff}
The Hausdorff dimension of a self-similar ultrametric Cantor set is equal to the abscissa of convergence of its {\it zeta}-function. 
\end{theo}
%%%%%%%%%%%%%%%%%%%%%
\begin{proof}
Since the Hausdorff dimension is invariant under bi-Lipschitz homeomorphisms, one can assume that the self-similar Cantor set $C=(\Pi_\infty, \rho)$ is  regular.
Denote by $\alpha$ the parameter of $\rho$.
We recall the definition of the Hausdorff dimension of $C$.
For $\delta>0, d\ge 0$ set 
%%%%%%%%%%%%%%%%%%%%%
\begin{equation}
\label{cantoremb10.eq-haus1}
H_\delta^d(C) = \inf \bigl\{
\sum_{i\in I} \diam(u_i)^d \, : \, (u_i)_{i\in I} \text{ is a covering by balls of diameters less than } \delta
\bigr\}\,,
\end{equation}
%%%%%%%%%%%%%%%%%%%%%
and define the $d$-dimensional Hausdorff content of $C$ to be
%%%%%%%%%%%%%%%%%%%%%
\begin{equation}
\label{cantoremb10.eq-haus2}
H^d(C) = \inf_{\delta} H^d_\delta(C) \,.
\end{equation}
%%%%%%%%%%%%%%%%%%%%%
The Hausdorff dimension is given by the critical  value of the Hausdorff content, namely:
%%%%%%%%%%%%%%%%%%%%%
\begin{equation}
\label{cantoremb10.eq-haus3}
\dim_H (C) = \inf \{ d \, : \, H^d(C) = 0 \} = \sup \{ d\, : \, H^d(C) > 0 \}\,.
\end{equation}
%%%%%%%%%%%%%%%%%%%%%

Consider the case $d> s_0$ first.
For any $n\in \NM$ there is a clopen partition of $C$ by the cylinders $[\gamma], \gamma \in \Pi_n$, so we have
%%%%%%%%%%%%%%%%%%%%%
\[
H^d(C) \le H^d_{\alpha^n}(C) \le \sum_{\gamma \in \Pi_n}  (\alpha^n)^d \le c \pf^n \alpha^{nd} = c \alpha^{n(d-s_0)}\,,
\]
%%%%%%%%%%%%%%%%%%%%%
where we used equation \eqref{cantoremb10.eq-s0} for the last equality.
This inequality holds for all $n$, and as $n$ tends to infinity the term \(\alpha^{n(d-s_0)}\) tends to zero. %so that $H^d(C)$ can be made arbitrarily small.
Hence \(H^d(C)=0\) for $d>s_0$, and this proves
%%%%%%%%%%%%%%%%%%%%%
\begin{equation}
\label{cantoremb10.eq-ineqhaus1}
 \dim_H (C) \le s_0  \,.
\end{equation}
%%%%%%%%%%%%%%%%%%%%% 

Consider now the case $d <s_0$.
Notice to begin with, that one can restrict to coverings that are finite clopen partitions of $C$, {\it i.e.\ } partitions by cylinders \([\gamma], \gamma\in\Pi\).
This is because, first: as $C$ is compact, one can consider finite coverings only, and second: since $\rho$ is a ultrametric, if two balls intersect then one is contained in the other.

Further, we will consider the $k$-th telescope \(C^{(k)}= (\Pi_\infty^{(k)}, \rho^{(k)})\) of $C$ as in Lemma~\ref{cantoremb10.lem-telescoping}, for an integer $k\ge 1$ to be determined later. 
%consider the $k$-th telescope \(C^{(k)}= (\Pi_\infty^{(k)}, \rho^{(k)})\) of $C$ as in Lemma~\ref{cantoremb10.lem-telescoping}.
For ay $k\ge 1$, $C^{(k)}$ is bi-Lipschitz equivalent to $C$, so they have the same Hausdorff dimension:
%%%%%%%%%%%%%%%%%%%%%
\begin{equation}
\label{cantoremb10.eq-haus4}
\forall k\ge 1\,, \quad \dim_H (C) = \dim_H(C^{(k)}) \,.
\end{equation}
%%%%%%%%%%%%%%%%%%%%% 
Consider a path $\gamma \in \Pi^{(k)}_n$, and let $\ext^{(k)}(\gamma)$ be the set of paths in $\Pi^{(k)}_{n+1}$ that extend $\gamma$: \(\eta \in \ext^{(k)}(\gamma)  \iff \eta_{[0,n]} = \gamma\).
The cardinality of $\ext^{(k)}(\gamma)$ equals the number of edges in $\Bb^{(k)}$ that extend $\gamma$ one generation further: it is given then by the sum of elements in a line of the $k$-th power of the adjacency matrix of $\Bb$, and behaves like $\pf^k$ as $k$ tends to infinity.
More precisely, there exists a constant \(c_{r(\gamma)}>0\) (that only depends on the range vertex of $\gamma$) such that one has
%%%%%%%%%%%%%%%%%%%%%
\begin{equation}
\label{cantoremb10.eq-haus6}
\# \ext^{(k)}(\gamma) = c_{r(\gamma)} \pf^k + o(\pf^k)\,.
\end{equation}
%%%%%%%%%%%%%%%%%%%%% 
Since $\diam([\gamma])= \alpha^{kn}$ we have
%%%%%%%%%%%%%%%%%%%%%
\[
\sum_{ \eta \in \ext^{(k)}(\gamma) } \; \diam([\eta])^d = \#  \ext^{(k)}(\gamma) \; (\alpha^{k(n+1)})^d =
( c_{r(\gamma)} \pf^k  + o(\pf^k) ) \alpha^{kd}  \; \diam([\gamma])^d \,, 
\]
%%%%%%%%%%%%%%%%%%%%%
where the last equality holds by equation~\eqref{cantoremb10.eq-haus6}.
Now, for $d<s_0$, equation \eqref{cantoremb10.eq-s0} implies that $\pf\alpha^d >1$.
So there exists $k_1$ such that for all $k\ge k_1$ the term \((c_{r(\gamma)} \pf^k  + o(\pf^k)  )\alpha^{kd}\) in the previous equation to be strictly greater than $1$.
And since the constant $c_{r(\gamma)}$ only depends on the range vertex of $\gamma$, and \(\# \Vv^{(k)} = \#\Vv\) does not depend on $k$, there exists $k_2\ge k_1$, such that for all $k\ge k_2$ this inequality holds for all paths $\gamma\in\Pi^{(k)}$.
Hence for $k\ge k_2$ we have
%%%%%%%%%%%%%%%%%%%%%
\begin{equation}
\label{cantoremb10.eq-haus5}
\forall \gamma \in \Pi^{(k)} \,, \quad \diam([\gamma])^d < \sum_{ \eta \in E^{(k)}(\gamma) } \diam([\eta])^d\,.
\end{equation}
%%%%%%%%%%%%%%%%%%%%%
We have proven that for any path $\gamma$, the diameter of its cylinder to the power $d$ is strictly smaller than the sum of those of its extensions.
This means that for any  $k\ge k_2$ the trivial covering of $C^{(k)}$ by itself (which corresponds to a path of length zero) minimizes the sum in equation~\eqref{cantoremb10.eq-haus1}.
Therefore, using equation~\eqref{cantoremb10.eq-haus4}, we deduce that for $k\ge k_2$
%%%%%%%%%%%%%%%%%%%%%
\[
H^d(C) = H^d(C^{(k)}) \ge \diam(C^{(k)})^d = 1 \quad \text{ for } d< s_0. 
\]
%%%%%%%%%%%%%%%%%%%%%
And with equation~\eqref{cantoremb10.eq-haus3} we conclude that
%%%%%%%%%%%%%%%%%%%%%
\begin{equation}
\label{cantoremb10.eq-ineqhaus2}
 \dim_H (C) \le s_0  \,.
\end{equation}
%%%%%%%%%%%%%%%%%%%%% 
Together with equation~\eqref{cantoremb10.eq-ineqhaus1} this completes the proof.
\end{proof}

%%%%%%%%%%%%%%%%%%%%%%%%%%%%%%%%%%%%%%%%%%%%%%%%%%%%%%%%%%%%%%%%%%%%
\section{Bi-Lipschitz embedding in $\RM^n$}
\label{cantoremb10.sect-Lipschitz}

%%%%%%%%%%%%%%%%%%%%%
\begin{theo}
\label{cantoremb10.thm-Lipschitz}
A self-similar ultrametric Cantor set is bi-Lipschitz embeddable in $\RM^n$, for some $n$.
\end{theo}
%%%%%%%%%%%%%%%%%%%%%
\begin{proof}
Consider a self-similar ultrametric Cantor set \((\Pi_\infty, \rho)\).
We can assume that $\rho$ is regular, with parameter \(\alpha \in (0,1)\).
Make a choice of a one-to-one mapping \(\beta : \Ee \rightarrow (0, +\infty)\), and define 
%%%%%%%%%%%%%%%%%%%%%
\begin{equation}
\label{cantoremb10.eq-deltabeta}
\dmi = \min_{e,f \in \Ee} | \beta(e) -\beta(f) | \,, \quad 
\dma = \max_{e,f \in \Ee} | \beta(e) -\beta(f) | \,.
\end{equation}
%%%%%%%%%%%%%%%%%%%%%
For an integer $n\in \NM$ we define a map \( F : \Pi_\infty \rightarrow \RM^n\), whose
$i$-th coordinate, \(1 \le i \le n\), is given by
%%%%%%%%%%%%%%%%%%%%%
\begin{equation}
\label{cantoremb10.eq-bilipemb}
F^i(x) = \sum_{j=0}^\infty \beta(x_{i + nj}) \alpha^{nj} \,,
\qquad 
x=(x_j)_{j\ge 0}  \in \Pi_\infty \,.
\end{equation}
%%%%%%%%%%%%%%%%%%%%%

Fix $x,y \in \Pi_\infty$, with \(|x\wedge y| = m\), so that we have \(\rho(x,y) = \alpha^{m}\).
We write \(m=nj_0 + i_0\), for some \(j_0 \in \NM\) and \( 0\le i_0 < n\).
We have the inequalities \( m < nj_0 \le m-n\), and get the bounds
%%%%%%%%%%%%%%%%%%%%%
\begin{equation}
\label{cantoremb10.eq-bdalpha}
\alpha^m < \alpha^{nj_0} \le \alpha^{-n} \alpha^m \,.
\end{equation}
%%%%%%%%%%%%%%%%%%%%% 
We consider $F^i$ and look at the case \(i \le i_0\) first.
All the terms in \(F^i(x)-F^i(y)\) vanish for \(j\le j_0\), so we get the upper bound
%%%%%%%%%%%%%%%%%%%%%
\[
|F^i(x)-F^i(y)| \le  \delta_{\max} \sum_{j\ge j_0+1} \alpha^{nj} =
\frac{\delta_{\max}}{1-\alpha^n} \alpha^{n (j_0+1)} \le \frac{\delta_{\max}}{1-\alpha^n} \alpha^m\,,   
\]
%%%%%%%%%%%%%%%%%%%%%
where we used equation~\eqref{cantoremb10.eq-bdalpha} for the last inequality.
We get the upper Lipschitz bound
%%%%%%%%%%%%%%%%%%%%%
\begin{equation}
\label{cantoremb10.eq-uplipbd1}
|F^i(x)-F^i(y)| \le \frac{\delta_{\max}}{1-\alpha^n} \alpha^m = c_1 \rho(x,y) \,, \quad i\le i_0 
\end{equation}
%%%%%%%%%%%%%%%%%%%%% 
For the lower bound we single out the first non-vanishing term in the series and write
%%%%%%%%%%%%%%%%%%%%%
\begin{equation}
\label{cantoremb10.eq-dimlip}
\begin{split}
|F^i(x)-F^i(y)| & \ge \delta_{\min} \alpha^{n(j_0+1)} - 
\Bigl| \sum_{j \ge j_0+ 2} (\beta(x_{i + nj}) - \beta(y_{i + nj})) \alpha^{nj} 
\Bigr| \\
& \ge \delta_{\min} \alpha^{n(j_0+1)} - \delta_{\max} \frac{ \alpha^{n(j_0+2)} }{1-\alpha^n}
=
\frac{\delta_{\min} - \alpha^n(\delta_{\min}+\delta_{\max})}{1-\alpha^n}
\alpha^{n (j_0+1)} \,.
\end{split}
\end{equation}
%%%%%%%%%%%%%%%%%%%%%
Since \(\alpha\in (0,1)\) we can find $n$ large enough for the numerator of the last fraction to be non-negative.
Using equation~\eqref{cantoremb10.eq-bdalpha} we get the lower Lipschitz bound
%%%%%%%%%%%%%%%%%%%%%
\begin{equation}
\label{cantoremb10.eq-lowlipbd1}
|F^i(x)-F^i(y)| \ge  \frac{\delta_{\min} - \alpha^n(\delta_{\min}+\delta_{\max})}{1-\alpha^n} \alpha^n \, \alpha^m
= c_2 \rho(x,y) \,, \quad i\le i_0 
\end{equation}
%%%%%%%%%%%%%%%%%%%%%

The case \(i> i_0\) is similar and yields the same bounds, but with constants multiplied by $\alpha^{-n}$.
That is the upper bound is the same as in equation \eqref{cantoremb10.eq-uplipbd1} but with constant
\(c_1 \alpha^{-n}\), and the lower bound is the same as in equation \eqref{cantoremb10.eq-lowlipbd1} but
with constant \(c_2 \alpha^{-n}\).
We let \(c_- = c_2 \) and \(c_+ = c_1 \alpha^{-n}\) and get the bounds
%%%%%%%%%%%%%%%%%%%%%
\begin{equation}
\label{cantoremb10.eq-lipbd1}
c_- \rho(x,y) \le |F^i(x)-F^i(y)| \le c_+ \rho(x,y) \,, \quad i = 1, \cdots n\,. 
\end{equation}
%%%%%%%%%%%%%%%%%%%%%
From this we deduce immediately the Lipschitz bounds for $F$ .
For the Euclidean norm, one gets for instance
%%%%%%%%%%%%%%%%%%%%%
\[
c_- \sqrt{n} \, \rho(x,y) \le \|F(x)-F(y)\| \le c_+ \sqrt{n} \, \rho(x,y) \,. 
\]
%%%%%%%%%%%%%%%%%%%%%
\end{proof}

We now give lower bounds on $n$.

%%%%%%%%%%%%%%%%%%%%%
\begin{lemma} 
\label{cantoremb10.lem-bilipdim}
Let $p = \# \Ee$, and set \(n = [-\log(p)/\log(\alpha)] + 1\).
Then \((\Pi_\infty, \rho)\) can be bi-Lipschitz embedded in $\RM^n$.
\end{lemma}
%%%%%%%%%%%%%%%%%%%%%
\begin{proof}
The condition on $n$ in equation \eqref{cantoremb10.eq-dimlip} in the proof of Theorem \ref{cantoremb10.thm-Lipschitz} reads
%%%%%%%%%%%%%%%%%%%%%
\[
n > \log \bigl(\frac{\delta_{\min}}{\delta_{\max}+\delta_{\min}}\bigr) 
/ \log(\alpha) \,.
\]
%%%%%%%%%%%%%%%%%%%%%
Now clearly, the minimum possible value of the quotient \(\dmi/(\dmi+\dma)\) is reached if one takes the constants $\beta(e)$ to be uniformly distributed, in which case it equals \(1/p\). 
\end{proof}

%%%%%%%%%%%%%%%%%%%%%
\begin{theo}
\label{cantoremb10.thm-hoelder2}
Let \(C=(\Pi_\infty,\rho)\) be a self-similar ultrametric Cantor set.
There exists a bi-Lipschitz embedding
%%%%%%%%%%%%%%%%%%%%%
\[
C \hookrightarrow \RM^{[\dim_H(C)]+1} \,,
\]
%%%%%%%%%%%%%%%%%%%%%
where $\dim_H(C)$ is the Hausdorff dimension of $C$, and $[\dim_H(C)]$ denotes its integer part.
\end{theo}
%%%%%%%%%%%%%%%%%%%%%
\begin{proof}
By Theorem~\ref{cantoremb10.thm-hausdorff} and Lemma~\ref{cantoremb10.lem-s0} we have $d_H = \dim_H(C) = -\log(\pf) / \log(\alpha)$.
According to Lemma~\ref{cantoremb10.lem-bilipdim}, the condition on $n$ for the lower Lipschitz bound in equation \eqref{cantoremb10.eq-dimlip} in the proof of Theorem~\ref{cantoremb10.thm-Lipschitz} reads 
%%%%%%%%%%%%%%%%%%%%%
\begin{equation}
\label{cantoremb10.eq-tac}
n > - \log(1+ \dma/\dmi) / \log(\alpha)  = d_H \log(p) / \log(\pf) \,.
\end{equation}
%%%%%%%%%%%%%%%%%%%%%
where $p$ denotes the cardinality of $\Ee$.
For an integer $k\ge 1$, consider the $k$-th telescope \(C^{(k)}= (\Pi_\infty^{(k)}, \rho^{(k)})\) of $C$ as in Lemma~\ref{cantoremb10.lem-telescoping}.
The parameter of $\rho^{(k)}$ is $\alpha^k$.
The cardinality $p^{(k)}$ of $\Ee^{(k)}$ grows like $\pf^k$ with $k$:
%%%%%%%%%%%%%%%%%%%%%
\begin{equation}
\label{cantoremb10.eq-pk}
p^{(k)}  = c \pf^k + o(\pf^k)\,,
\end{equation}
%%%%%%%%%%%%%%%%%%%%% 
where the constant $c$ is uniformly bounded (does not depend on $k$).
Hence the lower Lipschitz bound for $\Bb^{(k)}$ in equation \eqref{cantoremb10.eq-tac} reads
%%%%%%%%%%%%%%%%%%%%%
\[
n >  d_H \log(c\pf^k+ o(\pf^k)) / \log(\pf^k) = d_H + c'/k + o(1/k)\,, \quad c' = \log(c)/\log(\pf)\,.
\]
%%%%%%%%%%%%%%%%%%%%%
By Lemma \ref{cantoremb10.lem-telescoping}, $C$ is Lipschitz equivalent to \(C^{(k)}\) for all $k$.
Hence for any \( 0< \eps < 1\), by choosing $k$ large enough, this shows that $C$ is bi-Lipschitz embeddable in \(\RM^n\), for \(n> d_H + \eps\).
This completes the proof.
\end{proof}
As a corollary of this and Theorem~\ref{cantoremb10.thm-zeta} one gets the following.

%%%%%%%%%%%%%%%%%%%%%
\begin{theo} 
\label{cantoremb10.thm-bilipXi}
The transversal of a substitution tiling of $\RM^d$ is bi-Lipschitz embeddable in $\RM^{d+1}$.
\end{theo}
%%%%%%%%%%%%%%%%%%%%%

%%%%%%%%%%%%%%%%%%%%%%%%%%%%%%%%%%%%%%%%%%%%%%%%%%%%%%%%%%%%%%%%%%%%
\section{The Laplacians and their spectra}
\label{cantoremb10.sect-Laplacian}

Pearson and Bellissard built in \cite{PB09} a spectral triple (the data of Riemaniann noncommutative geometry \cite{Co94})
for ultrametric Cantor sets, and derived a family of Laplace-Beltrami like operators.
The authors revisited their construction in \cite{JS10} for the transversals of substitution tilings which are particular self-similar ultrametric Cantor sets.
We remind the reader here of the one parameter family of Laplace operators \((\Delta_s)_{s\in \RM}\) that one obtains in this case, and refer the reader to \cite{JS10} for the details.

Consider a substitution tiling of $\RM^d$, with primitive Abelianization matrix $A$ (see Example~\ref{cantoremb10.ex-substitution}, equation \eqref{cantoremb10.eq-abelmatrix}).
Denote by $\Xi$ the canonical transversal to its tiling space: it is a ultrametric Cantor set.
Write $\Lambda$ for its Perron-Frobenius eigenvalue, and \(\nu_i, i\in \Vv\), for the normalized coordinates of its associated eigenvector.
Let \(\Bb=(V,E)\) be the Bratteli diagram of the substitution, with ultrametric given as in Example~\ref{cantoremb10.ex-substitution}, equation \eqref{cantoremb10.eq-metricsubRd}.
One has a bi-Lipschitz homeomorphism between $\Pi_\infty$ and $\Xi$ \cite{JS10}.

The spectral triple is the data of a Dirac operator $D$ (self-adjoint, unbounded, with compact resolvent) on the Hilbert space $\ell^2(\Pi_\infty) \otimes \CM^2$ together with a faithful $\ast$-representation of the \Cs $C(\Pi_\infty)$ by bounded operators.
The {\em $\zeta$-function} of the spectral triple is the trace of $|D|^{-s}$:
%%%%%%%%%%%%%%%%%%%%%
\begin{equation}
\label{cantoremb10.eq-zeta}
\zeta(s) = \frac{1}{2} \TR \bigl( |D|^{-s} \bigr) 
= \sum_{\gamma \in \Pi} \diam[\gamma]^s \,,
\end{equation}
%%%%%%%%%%%%%%%%%%%%%
and its abscissa of convergence will be denoted by $s_0 \in \RMbar$.
The authors proved in \cite{JS10} that $s_0$ is equal to both $d$ and the exponent of complexity of the tiling.
This result is refined here.
As a corollary of Theorem~\ref{cantoremb10.thm-hausdorff} one gets the following.
%%%%%%%%%%%%%%%%%%%%%
\begin{theo}
\label{cantoremb10.thm-zeta}
Consider a primitive substitution tiling of $\RM^d$, with canonical transversal $\Xi$.
The abscissa of convergence $s_0$ of the {\it zeta}-function of $\Xi$ is equal to its Hausdorff dimension \(\dim_H(\Xi)\), and moreover one has
%%%%%%%%%%%%%%%%%%%%%
\begin{equation}
\label{cantoremb10.eq-s0tiling}
s_0 =\dim_H(\Xi) = d\,.
\end{equation}
%%%%%%%%%%%%%%%%%%%%%
\end{theo}
%%%%%%%%%%%%%%%%%%%%%
%\begin{proof}
%
%\end{proof}

The spectral triple allows to define a measure $\mdix$ on \(\Pi_\infty\) (Dixmier trace).
Using the Dirac operator one can define a one parameter family \((\Dd_s)_{s\in\RM}\) of closable Dirichlet form on \(L^2(\Pi_\infty, d\mdix)\).
For each $s$, $\Dd_s$ is associated with the generator of a Markov semi-group $\Delta_s$.
One has that $\Delta_s$ is a self-adjoint, definite, and non negative operator on \(L^2(\Pi_\infty, d\mdix)\).
For $s>s_0+2$, $\Delta_s$ is bounded.
And for $s\le s_0+2$, $\Delta_s$ is unbounded and has pure point spectrum.

We recall now the spectrum of $\Delta_s$.
The eigenvalues do not depend on $s$ and are parametrized by \(\Pi\). 
For \(\gamma\in\Pi\) the corresponding eigenspace is
%%%%%%%%%%%%%%%%%%%%%
\begin{equation}
\label{cantoremb10.eq-eigenspace}
E_\gamma = 
\Bigl< 
\frac{1}{\mdix[\gamma\cdot e]} \chi_{\gamma\cdot e} - 
\frac{1}{\mdix[\gamma\cdot e']} \chi_{\gamma\cdot e'} \, : \,
e,e' \in \ext(\gamma)
\Bigr>
\end{equation}
%%%%%%%%%%%%%%%%%%%%%
where \(\chi_{\eta}\) is the characteristic function of the cylinder $[\eta]$ of a path $\eta \in\Pi$, and \(\ext(\eta)\) is the set of edges that extend $\eta$ one generation further.
The dimension of the eigenspace is \(\dim E_\gamma = n_\gamma -1\), where \(n_\gamma = \# \ext(\gamma)\).
Now define
%%%%%%%%%%%%%%%%%%%%%
\begin{equation}
\label{cantoremb10.eq-Ls}
\Lambda_s = \pf^{(2+s_0-s)/s_0} = \alpha^{s-s_0-2}
\end{equation}
%%%%%%%%%%%%%%%%%%%%%
where $\pf$ is the Perron-Frobenius eigenvalue of $A$, and $\alpha$ is thus the parameter of the ultrametric on $\Pi_\infty$.
So as in Lemma~\ref{cantoremb10.lem-s0} one has \(s_0 = - \log(\pf) / \log(\alpha)\).
The self-similar structure allows to compute the eigenvalues explicitly, using a Cuntz-Krieger algebra associated with the diagram $\Bb$.
For a finite path \(\gamma=(e_0, e_1, \cdots e_n)\) in $\Bb$, one has its associated
eigenvalue $\lambda_\gamma(t)$ of $\Delta_t$, which one can write
%%%%%%%%%%%%%%%%%%%%%
\begin{equation}
\label{cantoremb10.eq-eigenvalue}
\lambda_\gamma(t) = \Lambda_t^n \lambda_{s(e_n)} + \sum_{j=1}^n \Lambda_t^{j-1} \beta(e_j,t) \,,
\end{equation}
%%%%%%%%%%%%%%%%%%%%%
where \(\lambda_{s(e_n)}\) is the eigenvalue associated with the path of length one from the root to $s(e_n)$ and does not depend on $t$, and where the constants \(\beta(e,t)\) depend only of the measures of the cylinders of $s(e)$ and $r(e)$, and are uniformly bounded.
In view of equation \eqref{cantoremb10.eq-Ls}, for \(s>s_0+2\) one has \(\Lambda_s <1\), thus as $n$ goes to infinity the first term in \eqref{cantoremb10.eq-eigenvalue} tends to zero, while the sum in the second term converges.
Given an infinite path \(x = (e_j)_{j\in\NM}\) one can thus define  $\lambda_x(s)$ as the sum
%%%%%%%%%%%%%%%%%%%%%
\begin{equation}
\label{cantoremb10.eq-lambdax}
\lambda_x(s) = \sum_{j=1}^\infty \beta(e_j,s)\Lambda_s^{j-1}\,, \quad \quad s>s_0+2\,.
\end{equation}
%%%%%%%%%%%%%%%%%%%%% 
One defines now, for $s>s_0+2$, the $\omega$-spectrum of $\Delta_s$ as
%%%%%%%%%%%%%%%%%%%%%
\[
\Sp_\omega (\Delta_s) = %\bigcap_{n \in \NM} \overline{\Sp(\Delta_s) \setminus \Sp \bigl( \Delta_s \vert_{\Pi_n} \bigr)} \,.
\bigcup_{n \in \NM} \overline{ \{ \lambda_k  :  k > n \} } \,,
\]
%%%%%%%%%%%%%%%%%%%%%
where the eigenvalues are ordered so that if $\lambda_k$ is an eigenvalue of $\Delta_s\vert_{\Pi_n}$ then so are all $\lambda_l$ for $l\le k$.
In the case where all eigenvalues are finite multiplicity, this coincides with the usual definition of the $\omega$-spectrum as the intersection of the spectra \(\overline{\Sp(\Delta_s) \setminus \Sp \bigl( \Delta_s \vert_{\Pi_n} \bigr)}\) over $n\in\NM$.
In our case here, this is the boundary of the pure point spectrum of $\Delta_s$.

Under some conditions on $A$ and the $\beta(e,s)$, and for $s$ large enough, it will now be shown that \(\Sp_\omega(\Delta_s)\) is bi-{\em  H\"older} homeomorphic to $\Pi_\infty$ (or $\Xi$), see Corollary \ref{cantoremb10.cor-holder}.
Using the explicit form of the constants $\beta(e,s)$ given in \cite{JS10}, one can derive the elementary following lemma.
%%%%%%%%%%%%%%%%%%%%%
\begin{lemma}
\label{cantoremb10.lem-tech}
Assume that the diagram satisfies the following: all edges are simple, and the coordinates of the Perron-Frobenius eigenvector of its adjacency matrix are pairwise distinct.
%if \(x,y\in\Pi_\infty\) are such that they go through the same vertices but differ on infinitely many edges, then $x=y$.
Then there exists $s_1 \in \RM$ such that for all $s>s_1$, if $e\ne f$ then either \(\beta(e,s) \ne \beta(f,s)\), or 
\(\beta(e',s) \ne \beta(f',s)\) for any \(e'\in \ext(e), f'\in\ext(f)\).
\end{lemma}
%%%%%%%%%%%%%%%%%%%%%

%%%%%%%%%%%%%%%%%%%%%
\begin{rem}
\label{cantoremb10.rem-tech}
The above condition in Lemma~\ref{cantoremb10.lem-tech} cannot be fulfilled {\em as is} if the Bratteli diagram is ``too symmetrical''.
Examples of such diagrams include the dyadic odometer which encodes the usual triadic Cantor set, and the diagram of the Thue-Morse substitution.
In such cases, the symmetry implies high degeneracies in the spectrum of $\Delta_s$ (in the dyadic case, the eigenvalues $\lambda_\gamma$ are all equal for $\gamma \in \Pi_n$).
Substitutions giving rise to such ``over symmetrical'' diagrams will not be considered for the purpose of H\"older embeddings of $\Xi$ using $\Delta_s$ in Corollaries~\ref{cantoremb10.cor-holder} and~\ref{cantoremb10.cor-holderdim}, and Lemma~\ref{cantoremb10.lem-holderdim}.
One can deal with such diagrams however by modifying the ultrametric with constants $a_\gamma, \gamma\in\Pi$, as in Definition~\ref{cantoremb10.def-sscantor}.
Although no longer regular, the modified ultrametric is Lipschitz equivalent by Remark~\ref{cantoremb10.rem-reg}, and this destroys the degeneracies of the eigenvalues.
\end{rem}
%%%%%%%%%%%%%%%%%%%%%

%%%%%%%%%%%%%%%%%%%%%%%%%%%%%%%%%%%%%%%%%%%%%%%%%%%%%%%%%%%%%%%%%%%%
\section{Bi-H\"older embedding in $\RM$}
\label{cantoremb10.sect-Hoelder}

%%%%%%%%%%%%%%%%%%%%%
\begin{theo}
\label{cantoremb10.thm-Hoelder}
A self-similar ultrametric Cantor set is bi-H\"older embeddable in the real line.
\end{theo}
%%%%%%%%%%%%%%%%%%%%%
\begin{proof}
Consider a self-similar ultrametric Cantor set \((\Pi_\infty, \rho)\).
We can assume that $\rho$ is regular, with parameter \(\alpha \in (0,1)\).
As in the proof of Theorem~\ref{cantoremb10.thm-Lipschitz}, make a choice of a one-to-one mapping \(\beta : \Ee \rightarrow (0, +\infty)\), and set as in equation \eqref{cantoremb10.eq-deltabeta}
%%%%%%%%%%%%%%%%%%%%%
\[
\dmi = \min_{e,f \in \Ee} | \beta(e) -\beta(f) | \,, \quad 
\dma = \max_{e,f \in \Ee} | \beta(e) -\beta(f) | \,.
\]
%%%%%%%%%%%%%%%%%%%%%
For \(s>0\), define the mapping \( \phi_s : \Pi_\infty \rightarrow \RM\) by
%%%%%%%%%%%%%%%%%%%%%
\[
\phi_s (x) = \sum_{j=0}^\infty \beta(x_j) \alpha^{sj}\,, \qquad 
x=(x_j)_{j\ge 0}  \in \Pi_\infty\,.
\]
%%%%%%%%%%%%%%%%%%%%%
The series converges since \( \alpha^s \in (0,1)\).
Fix $x,y \in \Pi_\infty$, with \(|x\wedge y| = m\), so that \(\rho(x,y) = \alpha^{m}\).
We have
%%%%%%%%%%%%%%%%%%%%%
\[
|\phi_s(x) -\phi_s(y) | \le \delta_{\max} \sum_{j\ge m} \alpha^{sj} 
= \frac{\delta_{\max}}{1-\alpha^s} \alpha^{s m}\,, 
\]
%%%%%%%%%%%%%%%%%%%%%
and with \eqref{cantoremb10.eq-Ls} we get the upper H\"older bound
%%%%%%%%%%%%%%%%%%%%%
\begin{equation}
\label{cantoremb10.eq-uphoelbd}
|\phi_s(x) -\phi_s(y) | \le c_+ \rho(x,y)^s \,. 
\end{equation}
%%%%%%%%%%%%%%%%%%%%%
For the other bound we single out the first non-vanishing term in the series and write
%%%%%%%%%%%%%%%%%%%%%
\[
\begin{split}
|\phi_s(x) -\phi_s(y) | & \ge \delta_{\min} \alpha^{s m} -
\Bigl| \sum_{j \ge m+1} (\beta(x_j) - \beta(y_j)) \alpha^{sj)} \Bigr| \\
& \ge \delta_{\min} \alpha^{s m} - \delta_{\max} \frac{\alpha^{s(m+1)}}{1-\alpha^s}
=
\frac{\delta_{\min} - \alpha^s(\delta_{\min}+\delta_{\max})}{1-\alpha^s}
\alpha^{s m} 
\end{split}
\]
%%%%%%%%%%%%%%%%%%%%%
and since \(\alpha \in (0,1) \) we can find $s$ large enough for the numerator of the last
fraction to be non-negative.
That is, for
%%%%%%%%%%%%%%%%%%%%%
\begin{equation}
\label{cantoremb10.eq-dimhoel}
s >  \log \bigl(\frac{\delta_{\min}}{\delta_{\max}+\delta_{\min}}\bigr) 
/ \log(\alpha) \,,
\end{equation}
%%%%%%%%%%%%%%%%%%%%% 
we get a lower H\"older bound
%%%%%%%%%%%%%%%%%%%%%
\begin{equation}
\label{cantoremb10.eq-lowhoelbd}
|\phi_s(x) -\phi_s(y) | \ge c_- \rho(x,y)^{s} \,. 
\end{equation}
%%%%%%%%%%%%%%%%%%%%% 
\end{proof}

%%%%%%%%%%%%%%%%%%%%%

\begin{coro} 
\label{cantoremb10.cor-holder}
Let $\Xi$ be the transversal of a substitution tiling of $\RM^d$.
If the Bratteli diagram of the substitution satisfies the condition of Lemma \ref{cantoremb10.lem-tech}, then for $s$ large enough $\Xi$ is bi-H\"older homeomorphic to the $\omega$-spectrum of $\Delta_s$.
\end{coro}
%%%%%%%%%%%%%%%%%%%%%
\begin{proof}
In this setting, the constant $\alpha$ in the proof of Theorem \ref{cantoremb10.thm-Hoelder} has to be replaced by
\(\pf^{(d+2-s)/d}\) as in equation \eqref{cantoremb10.eq-Ls} (here $s_0=d$ by Theorem~\ref{cantoremb10.thm-zeta}).
Fix $x,y \in \Pi_\infty$, with say \(|x\wedge y| = m\).
By Lemma~\ref{cantoremb10.lem-tech} for $s$ large enough, we have either \(\beta(x_m,s) \neq \beta(y_m,s)\) or \(
\beta(x_{m+1},s) \neq \beta(y_{m+1},s)\).
One uses the constants $\beta(e,s)$ to define the map $\phi_s$ in the proof of Theorem~\ref{cantoremb10.thm-Hoelder}, and one gets that \( |\phi_s(x)-\phi_s(y)|\) is bounded as in equations \eqref{cantoremb10.eq-uphoelbd} and \eqref{cantoremb10.eq-lowhoelbd} up to factors $\alpha^s$.
The condition on $s$ in equation \eqref{cantoremb10.eq-dimhoel} becomes here
%%%%%%%%%%%%%%%%%%%%%
\begin{equation}
\label{cantoremb10.eq-dimhoel2}
s > \max\{ s_1, d + 2 + d \log \bigl(1 + \dma/ \dmi \bigr) / \log(\pf) \}\,.
\end{equation}
%%%%%%%%%%%%%%%%%%%%% 
\end{proof}

In some cases, namely if $s_1$ is not too large, one can get a better lower bound on $s$ as follows.
%%%%%%%%%%%%%%%%%%%%%
\begin{lemma} 
\label{cantoremb10.lem-holderdim}
With the settings of the Corollary \ref{cantoremb10.cor-holder}, let $p=\# \Ee$, then $\Xi$ is bi-H\"older homeomorphic to the $\omega$-spectrum of $\Delta_s$ for all $s$ greater than \(d + 2 + d \log(p)/\log(\pf)\).
\end{lemma}
%%%%%%%%%%%%%%%%%%%%%
\begin{proof}
Here $p$ is the number of edges in the diagram between two generations away from the root.
The maximum possible value of \(1 + \dma/ \dmi\) in equation \eqref{cantoremb10.eq-dimhoel2} is reached when the $\beta(e)$ are uniformly distributed, in which case this equals $p$.
\end{proof}

%%%%%%%%%%%%%%%%%%%%%
\begin{coro} 
\label{cantoremb10.cor-holderdim}
If \(s_1 \le 2(d + 1)\), then $\Xi$ is bi-H\"older homeomorphic to the $\omega$-spectrum of $\Delta_s$ for all $s$ greater than \(2(d + 1)\).
\end{coro}
%%%%%%%%%%%%%%%%%%%%%
\begin{proof}
Using Lemma~\ref{cantoremb10.lem-holderdim} and telescoping $k$ times as in the proof of Theorem~\ref{cantoremb10.thm-hoelder2}, one gets that $\Xi$ is bi-H\"older embeddable for all $s$ greater than \(d + 2 + d\log(c\pf^k)/\log(\pf^k) = 2(d + 1) + c'/k\), and one can choose $k$ arbitrarily large.
\end{proof}

%%%%%%%%%%%%%%%%%%%%%%%%%%%%%%%%%%%%%%%%%%%%%%%%%%%%%%%%%%%%%%%%%%%%%%%%%%%

\end{document}